\documentclass[12pt,a4paper]{article}
\usepackage[margin=25mm]{geometry}
\usepackage{amssymb,amsthm,amsmath, amsfonts,ascmac, enumerate}
\numberwithin{equation}{section}
\usepackage[colorlinks=true,linkcolor=blue,citecolor=blue]{hyperref}

\theoremstyle{definition}
\newtheorem{thm}{Theorem}[section]
\newtheorem{defn}[thm]{Definition}
\newtheorem{lem}[thm]{Lemma}
\newtheorem{prop}[thm]{Proposition}
\newtheorem{rem}[thm]{Remark}
\newtheorem{cor}[thm]{Corollary}
\newtheorem{ex}[thm]{Example}
\newtheorem{prob}[thm]{Problem}

\newcommand{\Ran}{{\sf Ran}\ }
\newcommand{\Ker}{{\sf Ker}\ }
\newcommand{\Rank}{{\sf rank}}
\newcommand{\Det}{{\sf det}\ }
\newcommand{\diag}{{\sf diag}}
\newcommand{\Span}{{\sf span}}
\newcommand{\Dim}{{\sf dim}}

\title{Finite free position of maximal abelian $\ast$-subalgebras of the matrix algebra}

\author{Yuki Ueda}

\date{\today}

\begin{document}

\maketitle

\begin{abstract} 
Finite free convolution is obtained by averaging characteristic polynomials over Haar unitary conjugation. We ask when two maximal abelian $\ast$-subalgebras of the complex matrix algebra ${\cal M}_n$ can be placed in finite free position: that is, when their relative position realizes this averaging exactly for every pair of elements, one from each subalgebra.  Writing such a pair as ${\cal D}_n$ and $U{\cal D}_nU^*$ with $U$ unitary, we characterize finite free position, for either additive or multiplicative convolution, by the condition $|\Det U[I,J]|^2=\binom{n}{r}^{-1}$ for every $1\le r\le n$ and all $I,J$ with $|I|=|J|=r$. We show that this condition holds if and only if $n\le3$ and $\sqrt{n} U$ is a complex Hadamard matrix.

To quantify the failure of exact realization for $n\ge 4$, we introduce the uniform-minor discrepancy $\delta_r(U)$. We identify it with the mean-square error in the $r$-th coefficient of finite free multiplicative convolution for two diagonal matrices whose diagonal entries are independent and uniformly distributed on the unit circle. We establish the symmetry $\delta_r(U)=\delta_{n-r}(U)$ and the monotonicity $\delta_1(U)\le\delta_2(U)\le\cdots\le \delta_{\lfloor n/2\rfloor}(U)$. For flat unitaries, we derive an explicit formula for $\delta_2$, yielding $\delta_r(U)\ge\frac{n-3}{2n}$ for $2\le r\le n-2$. Equality for $r=2$ holds precisely when the entrywise square of $\sqrt{n}U$ is also complex Hadamard.
\end{abstract}

\vspace{3mm}



\section{Introduction}

Throughout this paper, ${\cal M}_{m\times n}$ denotes the space of $m\times n$ complex matrices, and we write ${\cal M}_n:={\cal M}_{n\times n}$. We denote by ${\cal U}_n$ the group of $n\times n$ unitary matrices. The matrices $I_n$ and $J_n$ denote the identity matrix and the all-ones matrix of order $n$, respectively.

Polynomial convolutions have a classical history in the work of Szeg\H{o} \cite{S22} and Walsh \cite{W22} on the preservation of real-rootedness. Building on this tradition, Marcus, Spielman and Srivastava \cite{MSS22} and Marcus \cite{Mar21} connected these convolutions with random matrices and free probability, developing a finite-dimensional theory known as {\it finite free probability}. The corresponding polynomial convolutions are now known as {\it finite free convolutions}. For $A,B\in{\cal M}_n$, the finite free additive and multiplicative convolutions of their characteristic polynomials $\chi_A$ and $\chi_B$ are defined by
$$
(\chi_A\boxplus_n\chi_B)(x):= \mathbb{E}_Q[\chi_{A+QBQ^\ast}(x)],
\qquad
(\chi_A\boxtimes_n\chi_B)(x):= \mathbb{E}_Q[\chi_{AQBQ^\ast}(x)],
$$
where $Q$ is Haar distributed on ${\cal U}_n$. Connections with free probability have been investigated from several perspectives; see, for example, \cite{AFPU26, AGP23, AP18, Fuj26, Mar21, MSS22}.

Following \cite{ALR25}, matrices $A,B\in{\cal M}_n$ are said to be in {\it additive finite free position} if
$$
\chi_A\boxplus_n\chi_B= \chi_{A+B},
$$
and in {\it multiplicative finite free position} if
$$
\chi_A\boxtimes_n\chi_B = \chi_{AB}.
$$
We say that two $\ast$-subalgebras are in additive (resp.\ multiplicative) finite free position if every pair of elements, one from each algebra, is in the corresponding finite free position. We ask when unitary averaging in finite free convolution can be replaced by a single deterministic relative position, simultaneously for all such pairs of elements. In this paper, we address this question for maximal abelian $\ast$-subalgebras (MASAs) of ${\cal M}_n$.
Let
$$
{\cal D}_n :=\{\diag(d_1,\ldots,d_n):d_i\in\mathbb{C}\}.
$$
Every pair of MASAs is, up to simultaneous unitary conjugation, of the form ${\cal D}_n$ and $U{\cal D}_nU^\ast$ for some $U\in{\cal U}_n$. Thus our problem takes the following concrete form.

\begin{prob}\label{prob:main}
For which $n\in\mathbb{N}$ and $U\in{\cal U}_n$ are ${\cal D}_n$ and $U{\cal D}_nU^\ast$ in additive (resp.\ multiplicative) finite free position? Equivalently, when does
$$
\chi_A\boxplus_n\chi_B=\chi_{A+UBU^\ast}
\qquad
\left(\text{resp. }\;
\chi_A\boxtimes_n\chi_B=\chi_{AUBU^\ast}\right)
$$
hold for every $A,B\in{\cal D}_n$?
\end{prob}
Arizmendi, Lehner and Rosenmann \cite{ALR25} identified the finite free complement of ${\cal D}_n$, for both convolutions, with the class of principally balanced matrices: matrices whose principal minors of each fixed order are all equal. Consequently, Problem~\ref{prob:main} is equivalent to determining those $U\in{\cal U}_n$ for which every matrix in $U{\cal D}_nU^\ast$ is principally balanced.

Our first main result gives a complete answer. Write $[n]:=\{1,\ldots,n\}$ and $N_{n,r}:=\binom{n}{r}$, and let $U[I,J]$ denote the submatrix with row set $I$ and column set $J$. Theorem~\ref{thm:MASA} shows that the following conditions are equivalent:
\begin{itemize}
\item ${\cal D}_n$ and $U{\cal D}_nU^\ast$ are in additive finite free position;
\item ${\cal D}_n$ and $U{\cal D}_nU^\ast$ are in multiplicative finite free position;
\item $|\det U[I,J]|^2=N_{n,r}^{-1}$ for every $1\le r\le n$ and every $I,J\subset[n]$ with $|I|=|J|=r$;
\item $n\le3$ and $\sqrt{n}U$ is a complex Hadamard matrix.
\end{itemize}
Here a matrix $H=(h_{ij})$ is said to be {\it complex Hadamard} if $HH^\ast=nI_n$ and $|h_{ij}|=1$ for all $i,j$.
We call $U$ {\it flat} when $\sqrt{n}U$ is complex Hadamard. The uniform-minor condition has a natural interpretation through exterior powers: the $(I,J)$-entry of $\bigwedge^r U$ is $\det U[I,J]$. It therefore requires every exterior power of $U$ to have entries of equal modulus. For $r=1$, this is precisely flatness. Flatness together with the condition for $r=2$ already forces $n\le3$, so exact finite free position of two MASAs is impossible in every dimension $n\ge4$.

To measure this obstruction, we introduce the {\it uniform-minor discrepancy}
$$
\delta_r(U) :=\sum_{\substack{I,J\subset[n]\\|I|=|J|=r}}\left(|\det U[I,J]|^2-\frac{1}{N_{n,r}}\right)^2,
\qquad 1\le r\le n,
$$
and set $\delta_0(U):=0$. Thus $\delta_r(U)=0$ precisely when the minors of order $r$ have the uniform modulus required by finite free position. Our second main result describes the structure of these discrepancies across different orders. Set
$$
P_r(U):=\left(|\det U[I,J]|^2\right)_{|I|=|J|=r}.
$$
These matrices are doubly stochastic and satisfy $D_rP_r(U)=P_{r-1}(U)D_r$, where $D_r$ is the incidence matrix between $(r-1)$-subsets and $r$-subsets of $[n]$. For the centered matrices
$$
\Delta_r(U):= P_r(U)-\frac{1}{N_{n,r}}J_{N_{n,r}},
$$
we obtain an orthogonal decomposition
$$
\Delta_r(U)\simeq\Delta_{r-1}(U)\oplus R_r(U), \qquad 2\le r\le\lfloor n/2\rfloor,
$$
for a suitable operator $R_r(U)$, where $\simeq$ denotes unitary equivalence.  Since $\delta_r(U)=\|\Delta_r(U)\|_F^2$, where $\|A\|_F:=\left(\sum_{i,j}|a_{ij}|^2\right)^{1/2}$ is the Frobenius norm of a matrix $A=(a_{ij})$, this decomposition gives
$$
\delta_r(U)=\delta_{r-1}(U)+\|R_r(U)\|_F^2,
\qquad 2\le r\le\lfloor n/2\rfloor.
$$
Consequently, together with complementary symmetry, we obtain in Theorem~\ref{thm:monotonicity} that
$$
\delta_1(U)\le\delta_2(U)\le\cdots\le
\delta_{\lfloor n/2\rfloor}(U)
\qquad \text{and} \qquad
\delta_r(U)=\delta_{n-r}(U).
$$

The discrepancy also has a direct interpretation as an error in finite free multiplicative convolution.
Let $z_1,\ldots,z_n$ and $w_1,\ldots,w_n$ be independent Haar random variables on $\mathbb{T}$, and set $E_z:=\diag(z_1,\ldots,z_n)$ and $E_w:=\diag(w_1,\ldots,w_n)$.
Writing $\chi_A(x)=\sum_{r=0}^n(-1)^r {\sf e}_r(A)x^{n-r}$, we prove in Theorem~\ref{thm:L2-interpretation} that
$$
\delta_r(U) = \mathbb{E}\left|{\sf e}_r(E_zUE_wU^\ast)-\frac{{\sf e}_r(E_z){\sf e}_r(E_w)}{N_{n,r}}\right|^2.
$$
The second term inside the absolute value is the coefficient of $x^{n-r}$ in $\chi_{E_z}\boxtimes_n\chi_{E_w}$,
up to the sign $(-1)^r$. Thus $\delta_r(U)$ is exactly the mean-square error in the coefficient of $x^{n-r}$ for the fixed relative position $U$.

Finally, we obtain a sharp quantitative bound for flat unitaries. For $n\ge4$ and $H:=\sqrt{n}U$ complex Hadamard, we obtain in Theorem~\ref{thm:UMD} that
$$
\delta_2(U) = \frac{n-3}{2n} + \frac{1}{2n^4} \left\|H^{\circ2}(H^{\circ2})^\ast-nI_n\right\|_F^2,
$$
where $H^{\circ2}$ denotes the entrywise square.
Consequently,
$$
\delta_2(U)\ge\frac{n-3}{2n},
$$
with equality precisely when $H^{\circ2}$ is also complex Hadamard, that is, when $H$ is an S-Hadamard matrix \cite{L19}.
Together with monotonicity and complementary symmetry, this yields
$$
\delta_r(U)\ge\frac{n-3}{2n}, \qquad 2\le r\le n-2.
$$
For $r=2$, normalized Fourier matrices attain the lower bound in every odd dimension $n\ge5$. In dimension $4$, we determine the exact minimum over flat unitaries and obtain $\min_{U\in{\cal U}_4^{\rm flat}}\delta_2(U)=1/4$, which is strictly larger than the universal lower bound. We also compute $\delta_2$ for Fourier matrices of even order and derive its exact behaviour under tensor products.

The paper is organized as follows. In Section~\ref{sec2}, we establish the exterior-power and incidence identities. In Section~\ref{sec3}, we prove the classification of finite free position for pairs of MASAs. In Section~\ref{sec4}, we develop the discrepancy, its monotonicity, and its interpretation as a coefficient error. In Section~\ref{sec5} we prove the explicit formula for $\delta_2$ and study its equality cases and extremal behaviour.


\section{Exterior powers and incidence relations}\label{sec2}

Let $\{e_i\}_{i=1}^n$ be the standard orthonormal basis of $\mathbb{C}^n$.
For $1\le r\le n$ and $I=\{i_1<i_2<\cdots<i_r\}\subset[n]$, we define an $r$-vector
$$
\bigwedge_{i\in I} e_i:=e_{i_1}\wedge e_{i_2}\wedge\cdots\wedge e_{i_r},
$$
where $\wedge$ is the exterior product. 
The $r$-th exterior power of $\mathbb{C}^n$ is defined by
$$
\bigwedge^r\mathbb{C}^n:=\Span\left\{\bigwedge_{i\in I} e_i:I\subset[n],\ |I|=r \right\}.
$$
Equipped with the standard inner product, $\bigwedge^r\mathbb{C}^n$ has dimension $\Dim\bigwedge^r\mathbb{C}^n=N_{n,r}$ and $\{\bigwedge_{i\in I} e_i:I\subset[n],\ |I|=r\}$ forms an orthonormal basis.

For $U\in\mathcal U_n$, we define the $r$-th exterior power $\bigwedge^r U$ by
$$
\left(\bigwedge^r U\right)
(v_1\wedge\cdots\wedge v_r):=Uv_1\wedge\cdots\wedge Uv_r,
\qquad v_1,\ldots,v_r\in\mathbb{C}^n.
$$
Since $U$ is unitary, $\bigwedge^r U$ is also unitary. Moreover, its matrix representation with respect to the above orthonormal basis is given as follows.

\begin{prop}\label{prop:exterior}
Let $U=(u_{ij})_{i,j=1}^n\in\mathcal U_n$.
For every $J\subset[n]$ with $|J|=r$, we have
\begin{align}\label{eq:rep}
\left(\bigwedge^r U\right)\bigwedge_{j\in J} e_j
=\sum_{\substack{I\subset[n]\\ |I|=r}} \Det U[I,J]\,\bigwedge_{i\in I} e_i.
\end{align}
\end{prop}
\begin{proof}
Let $J=\{j_1<j_2<\cdots<j_r\}\subset[n]$. Since $Ue_{j_k}=\sum_{i=1}^n u_{ij_k}e_i$ for each $1\le k\le r$, we have
\begin{align*}
\left(\bigwedge^r U\right)\bigwedge_{j\in J} e_j
=\sum_{i_1,\ldots,i_r=1}^n
u_{i_1j_1}\cdots u_{i_rj_r}
\,e_{i_1}\wedge\cdots\wedge e_{i_r}.
\end{align*}
If $i_k=i_\ell$ for some $k\neq \ell$, then $e_{i_1}\wedge\cdots\wedge e_{i_r}=0$. Thus, only the terms with pairwise distinct indices contribute.

Fix $I=\{i_1<i_2<\cdots<i_r\}\subset[n]$. By the alternating property of the exterior product, the coefficient of $\bigwedge_{i\in I} e_i$ in the above expansion is
\begin{align*}
\sum_{\sigma\in S_r}
{\sf sgn}(\sigma)
\prod_{k=1}^r u_{i_{\sigma(k)}j_k}=\Det (u_{i_pj_q})_{1\le p,q\le r}=\Det U[I,J].
\end{align*}
Therefore, the matrix representation \eqref{eq:rep} is obtained.
\end{proof}

By Proposition~\ref{prop:exterior}, with respect to the orthonormal basis $\{\bigwedge_{i\in I} e_i:I\subset[n],\ |I|=r\}$ of $\bigwedge^r\mathbb{C}^n$, the $(I,J)$-entry of $\bigwedge^r U$ is $\Det U[I,J]$. For any $U\in  {\cal U}_n$ and $1\le r\le n$, we define
\begin{align*}
P_r(U):=(|\Det U[I,J]|^2)_{|I|=|J|=r} \in {\cal M}_{N_{n,r}},
\end{align*}
and $P_0(U):=(1)$. By Proposition~\ref{prop:exterior}, we have
$$
P_r(U)=\left|\bigwedge^rU\right|^{\circ2},
$$
where the absolute value is taken entrywise. Since $\bigwedge^rU$ is unitary, the matrix $P_r(U)$ is doubly stochastic, i.e.
\begin{align}\label{eq:doubly}
\sum_{|I|=r}|\Det U[I,J]|^2=\sum_{ |J|=r} |\Det U[I,J]|^2=1.
\end{align}

For $2\le r\le n$, define the inclusion matrix $D_r \in {\cal M}_{N_{n,r-1} \times N_{n,r}}$ by
\begin{align}\label{eq:D}
(D_r)_{A,I} := \begin{cases}
1, & A\subset I\\
0, & A \not\subset I,
\end{cases}
\qquad |A|=r-1, \ |I|=r,
\end{align}
and $D_1:=\begin{pmatrix}
1 & 1 & \cdots & 1
\end{pmatrix} \in {\cal M}_{1\times n}$. The matrices $P_r(U)$ satisfy a natural consistency relation.
\begin{lem}\label{lem:DP}
For any $1\le r \le n$ and $U\in  {\cal U}_n$, we have
$$
D_rP_r(U)=P_{r-1}(U)D_r.
$$
\end{lem}
\begin{proof}
First, we consider $r=1$. Then 
\begin{align*}
(D_1P_1(U))_{1,j}=\sum_{i=1}^n |u_{ij}|^2 = 1 = (P_0(U) D_1)_{1,j}
\end{align*}
for all $j=1,\dots, n$. Now suppose that $2\le r \le n$. Fix $A, J \subset[n]$ with $|A|=r-1$ and $|J|=r$. The $(A,J)$-entry of $D_rP_r(U)$ is given by
\begin{align*}
(D_rP_r(U))_{A,J}=\sum_{|I|=r} (D_r)_{A,I} (P_r(U))_{I,J} = \sum_{\substack{|I|=r\\ A\subset I}} |\Det U[I,J]|^2.
\end{align*}
On the other hand, we have
\begin{align*}
(P_{r-1}(U)D_r)_{A,J}= \sum_{|B|=r-1} (P_{r-1}(U))_{A,B} (D_r)_{B,J}=\sum_{\substack{|B|=r-1\\B\subset J}}|\Det U[A,B]|^2.
\end{align*}
We prove the equality as follows. Let us consider $J=\{j_1<\cdots < j_r\}$ and define $V:=U[[n],J]\in {\cal M}_{n\times r}$. Since $U$ is unitary, we have $V^\ast V=I_r$. Moreover, we define $W:=V[A,[r]]=U[A,J]\in {\cal M}_{(r-1)\times r}$. A matrix $W_{\overline{k}} \in {\cal M}_{r-1}$ denotes the submatrix of $W$ obtained by deleting its $k$-th column. Note that $W_{\overline{k}}=U[A, J\setminus \{j_k\}]$. Define $c_k:=(-1)^{r+k}\ \Det W_{\overline{k}}$ for all $k=1,\dots ,r$ and $c=(c_1,\dots, c_r)^T \in \mathbb{C}^r$.
If we write the $i$-th row of $V$ as $x_i=(x_{i,1},\dots, x_{i,r}):=(u_{i,j_1},\dots, u_{i,j_r})$ for each $i\in [n]$, then
$$
\Det W[x_i]= \sum_{k=1}^r (-1)^{r+k}x_{i,k} \ \Det W_{\overline{k}} = x_i c, \quad \text{where} \quad W[x_i]:=\begin{pmatrix}
W\\
x_i
\end{pmatrix} \in {\cal M}_r.
$$
If $i\notin A$, then 
$$
|\Det U[A\cup \{i\},J]| = |\Det W[x_i]| = |x_ic|.
$$
If $i\in A$, then two rows of $W[x_i]$ coincide, and therefore $x_i c=0$. Consequently, we have
\begin{align*}
\sum_{\substack{|I|=r\\ A\subset I}}|\Det U[I,J]|^2 = \sum_{i\notin A} |x_ic|^2 = \sum_{i=1}^n |x_i c|^2=\|Vc\|^2 = \|c\|^2,
\end{align*}
where the last equality holds since $V^\ast V=I_r$. By definition, we have
\begin{align*}
\|c\|^2 &= \sum_{k=1}^r |c_k|^2 = \sum_{k=1}^r |\Det W_{\overline{k}}|^2\\
&=\sum_{k=1}^r |\Det U[A, J\setminus \{j_k\}]|^2 = \sum_{\substack{|B|=r-1\\B\subset J}}|\Det U[A,B]|^2,
\end{align*}
as desired.
\end{proof}


\section{Finite free position of MASAs}\label{sec3}
In this section, we give a characterization of unitary matrices to answer Problem~\ref{prob:main}. 

For $A\in {\cal M}_n$ and $I,J\subset [n]$ with $|I|=|J|$, the determinant $\Det A[I,J]$ is called a {\it minor} of $A$. In particular, it is called a {\it principal minor} of $A$ if $I=J$. A matrix $A\in {\cal M}_n$ is said to be {\it principally balanced} if for each $1\le r\le n$, all principal minors of order $r$ are equal, that is, $\Det A[I,I]$ does not depend on $I \subset [n]$ with $|I|=r$. Let ${\cal B}_n$ be the set of all $n\times n$ principally balanced matrices; see \cite{HO17} for details. In \cite[Theorems 4.7 and 6.12]{ALR25}, Arizmendi, Lehner and Rosenmann showed that the additive and multiplicative finite free complements of ${\cal D}_n$ both coincide with ${\cal B}_n$. Thus, a matrix is in additive (respectively, multiplicative) finite free position with every diagonal matrix if and only if it is principally balanced. Consequently, ${\cal D}_n$ and $U{\cal D}_nU^*$ are in additive (respectively, multiplicative) finite free position if and only if
\begin{align}\label{prob:main2}
U{\cal D}_nU^*\subseteq{\cal B}_n.
\end{align}
For $n\geq 2$, the inclusion $U{\cal D}_nU^*\subseteq{\cal B}_n$ does not hold for every $U\in{\cal U}_n$. Indeed, if $U=I_n$, then $U{\cal D}_nU^*={\cal D}_n$, which is not contained in ${\cal B}_n$. To answer Problem~\ref{prob:main} (equivalently, \eqref{prob:main2}), we denote by ${\cal H}_n$ the set of all $n\times n$ complex Hadamard matrices.
\begin{thm}\label{thm:MASA}
The following conditions are equivalent for $n\in \mathbb{N}$ and $U\in {\cal U}_n$.
\begin{enumerate}[\rm (1)]
\item $U{\cal D}_nU^\ast \subseteq{\cal B}_n$.
\item For any $1\le r \le n$ and for any $I,J\subset [n]$ with $|I|=|J|=r$, we have
\begin{align}\label{eq:uniform-minor}
|\Det U[I,J]|^2 = \frac{1}{N_{n,r}}.
\end{align}
\item For any $1\le r \le n$, we have $\sqrt{N_{n,r}} \left(\bigwedge^r U \right) \in {\cal H}_{N_{n,r}}$.
\item $n\le 3$ and $\sqrt{n}U \in{\cal H}_n$.
\end{enumerate}
\end{thm}

Before proving Theorem~\ref{thm:MASA}, we prepare a standard result which follows from the Cauchy--Binet formula.

\begin{lem}\label{lem:A}
For $U\in  {\cal U}_n$ and $D=\diag(d_1,\dots, d_n) \in {\cal D}_n$, we write $B=UDU^\ast$. Then, for any $I\subset [n]$ with $|I|=r$, we have
\begin{align*}
\Det B[I,I] = \sum_{|J|=r} |\Det U[I,J]|^2 d_J,
\end{align*}
where $d_J:=\prod_{j\in J} d_j$.
\end{lem}
\begin{proof}
It is easy to see that $B[I,I] = U[I,[n]] D U[I,[n]]^\ast$.
By the Cauchy--Binet  formula (see e.g. \cite[(3.14)]{Tao12}), we obtain
\begin{align*}
\Det B[I,I]  
&= \sum_{|J|=r} \left( \Det U[I,J] \right) \left( \Det D[J,J] \right) \left( \Det U[I,J]^\ast \right)\\
&= \sum_{|J|=r} |\Det U[I,J]|^2 \prod_{j\in J} d_j.
\end{align*}
\end{proof}

We now prove Theorem~\ref{thm:MASA}.

\begin{proof}[Proof of Theorem~\ref{thm:MASA}]
{\bf (1) $\Leftrightarrow$ (2):}
First, we assume that $U {\cal D}_n U^\ast \subseteq {\cal B}_n$ for $n\in \mathbb{N}$ and $U\in  {\cal U}_n$. For any $D=\diag(d_1,\dots, d_n)\in {\cal D}_n$, we have $B=UDU^\ast \in {\cal B}_n$. Therefore, for any $I_1,I_2 \subset [n]$ with $|I_1|=|I_2|=r$, we get $\Det B[I_1,I_1] =\Det B[I_2,I_2]$. By Lemma~\ref{lem:A}, this means that
$$
\sum_{|J|=r} |\Det U[I_1,J]|^2 d_J = \sum_{ |J|=r} |\Det U[I_2,J]|^2 d_J.
$$
Since $d_1,\dots, d_n$ are independent variables and the squarefree monomials $d_J$ are distinct, comparison of coefficients yields $|\Det U[I_1,J]|=|\Det U[I_2,J]|$ for any $J\subset [n]$ with $|J|=r$. By \eqref{eq:doubly}, we obtain
$$
|\Det U[I,J]|^2 = \frac{1}{N_{n,r}}
$$
for any $I, J \subset [n]$ with $|I|=|J|=r$. 

Conversely, we assume that the statement (2) holds. By Lemma~\ref{lem:A}, for any $1\le r\le n$ and $I\subset [n]$ with $|I|=r$, we have
\begin{align*}
\Det B[I,I] = \sum_{|J|=r} |\Det U[I,J]|^2 d_J = \frac{1}{N_{n,r}}\sum_{|J|=r} d_J,
\end{align*}
which does not depend on $I$. Hence $B\in {\cal B}_n$.\\

\hspace{-6mm}{\bf (2) $\Leftrightarrow$ (3):} It follows from Proposition~\ref{prop:exterior} and the definition of complex Hadamard matrices. \\

\hspace{-6mm}{\bf (2) $\Leftrightarrow$ (4):} First, assume that (2) holds. The case $n=1$ is trivial. Suppose now that $n\ge 2$. For any $I=\{i\}, J=\{j\} \subset [n]$, we have $|u_{ij}|^2 = |\Det U[I,J]|^2 = n^{-1}$.
Therefore, we have $|u_{ij}|=n^{-1/2}$, and also $|\sqrt{n} u_{ij}| = 1$ for any $1\le i,j\le n$. Since $U$ is unitary, we get $\sqrt{n} U \in {\cal H}_n$. We denote by $e^{{\rm i}\theta_k}/\sqrt{n}$ $(\theta_k\in \mathbb{R})$ the $(1,k)$-entry of $U$ for each $1\le k \le n$. Without loss of generality, we may assume that every entry in the first row of $U$ is $1/\sqrt{n}$ since $|\Det U[I,J]| =\left|{\sf det}((UE)[I,J])\right|$ for any $I,J\subset [n]$ with $|I|=|J|$, where $E=\diag(e^{-{\rm i}\theta_1},e^{-{\rm i}\theta_2},\dots, e^{-{\rm i}\theta_n})$. We denote by $z_k/\sqrt{n}$ ($|z_k|=1$) the $(2,k)$-entry of $U$ for each $1\le k \le n$. If $I=\{1,2\}$ and $J=\{j,k\}$ for some $j,k\in [n]$ with $j\neq k$, then
$$
|\Det U[I,J]|^2 = \frac{1}{n^2}|z_j-z_k|^2.
$$
Since the condition (2) holds true, we have $\frac{1}{n^2}|z_j-z_k|^2 = \binom{n}{2}^{-1}$.
This implies that,
\begin{align}\label{eq:zkzj}
|z_j-z_k| = \sqrt{\frac{2n}{n-1}}, \qquad j,k\in[n],\quad j\ne k.
\end{align}
Equation \eqref{eq:zkzj} forces $n\le 3$. Indeed, if we write $d= \sqrt{\frac{2n}{n-1}}$ and $v_j= z_j- z_n$ for $j=1,\dots, n-1$, then $\|v_j\|^2 =d^2$. For $j\neq k$, we have
$$
d^2 = \|v_j-v_k\|^2 = 2d^2 - 2\langle v_j , v_k\rangle,
$$
and therefore $\langle v_j, v_k\rangle =d^2/2$, where $\langle \cdot, \cdot \rangle$ is the real Euclidean inner product on $\mathbb{R}^2$. We define $G=(\langle v_i, v_j\rangle)_{1\le i,j\le n-1}$ as the Gram matrix for $\{v_1,\dots, v_{n-1}\}$. Then we can write
$$
G= \frac{d^2}{2} (I_{n-1} + J_{n-1}).
$$
A simple computation shows that $\Rank (I_{n-1}+J_{n-1})=n-1$, and hence $\Rank  (G)= n-1$. Since $G$ is a Gram matrix for $\{v_1,\dots, v_{n-1}\}\subset \mathbb{C} \simeq \mathbb{R}^2$, we must have $\Rank (G) \le 2$. Therefore $n-1 \le 2$, which leads to $n\le 3$.

Next, we assume that $n\le 3$ and $\sqrt{n} U \in {\cal H}_n$. The case $n=1$ is trivial. 
\begin{itemize}
\item {\bf The case $n=2$:} it suffices to consider the case $r=1$. Since $\sqrt{2}U\in {\cal H}_2$, we have $|\Det U[I,J]|^2=|u_{ij}|^2 =  \binom{2}{1}^{-1}$ for any $I=\{i\}, J=\{j\} \subset [2]$. 
\item {\bf The case $n=3$:} it suffices to consider the cases $r=1$ and $r=2$. First consider $r=1$. Since $\sqrt{3}U\in {\cal H}_3$, we have $|\Det U[I,J]|^2=|u_{ij}|^2 = \binom{3}{1}^{-1}$ for any $I=\{i\}, J=\{j\}\subset [3]$. Next, we consider the case $r=2$. Suppose that $I,J\subset[3]$ satisfy $|I|=|J|=2$. By Jacobi's complementary minor identity and the unitarity of $U$, we have $|\Det U[I,J]|=|\Det U[I^c,J^c]|$. Consequently, the case $r=2$ reduces to the case $r=1$.
\end{itemize}
\end{proof}


\section{Uniform-minor discrepancy}\label{sec4}

By Theorem~\ref{thm:MASA}, finite free position between ${\cal D}_n$ and $U{\cal D}_nU^*$ is impossible for $n\geq 4$. This motivates us to introduce a quantitative measure of the failure of finite free position.

\begin{defn}\label{defn:UMD}
For $0\le r \le n$, we define the mapping $\delta_r:  {\cal U}_n \to [0,\infty)$ by
$$
\delta_r(U):= \sum_{\substack{I,J \subset [n]\\ |I|=|J|=r}} \left( |\Det U[I,J]|^2 - \frac{1}{N_{n,r}}\right)^2, \qquad U\in  {\cal U}_n,\  1\le r \le n
$$
and $\delta_0(U):= 0$. We call $\delta_r(U)$ the $r$-th {\it uniform-minor discrepancy} of $U$. 
\end{defn}

\begin{rem}
By Theorem~\ref{thm:MASA}, the condition $\delta_r(U)=0$ for every $0\leq r\leq n$ is equivalent to the condition that ${\cal D}_n$ and $U{\cal D}_nU^*$ are in finite free position.
\end{rem}

\subsection{Basic properties}
We record several elementary properties of the discrepancy. Define
$$
 {\cal U}_n^{\rm flat}:=\{U\in  {\cal U}_n: \sqrt{n}U \in {\cal H}_n\} \subset  {\cal U}_n.
$$
We call matrices in ${\cal U}_n^{\rm flat}$ {\it flat unitary matrices}.

\begin{lem}\label{lem:C}
Let $U\in  {\cal U}_n$.
\begin{enumerate}[\rm (1)]
\item $\delta_1(U)=0$ if and only if $U \in  {\cal U}_n^{\rm flat}$.
\item For $0\le r \le n$, we have $\delta_r(U)=\delta_{n-r}(U)$.
\item For $1\le r \le n$, we have $ \delta_r(U) = \sum_{ |I|=|J|=r}  |\Det U[I,J]|^4 -1$.
\item If $U_1,U_2$ are diagonal unitary matrices and $P_1,P_2$ are permutation matrices, then $\delta_r(U_1P_1UP_2U_2)=\delta_r(U)$ for every $0\le r\le n$.
\end{enumerate}
\end{lem}
\begin{proof}
\begin{enumerate}[\rm (1)]
\item It is clear by definition.
\item Consider $1\le r \le n-1$. Since $|\Det U[I,J]|=|\Det U[I^c,J^c]|$ for any $I,J\subset [n]$ with $|I|=|J|=r$, the desired result holds by definition. If $r=n$, then $I=J=[n]$, and therefore $|\Det U[I,J]|=|\Det U|=1$. Thus $\delta_n(U)=0=\delta_0(U)$.
\item Since $\sum_{ |I|=|J|=r}  |\Det U[I,J]|^2=N_{n,r}$, we have
\begin{align*}
\delta_r(U)  
&= \sum_{ |I|=|J|=r}  |\Det U[I,J]|^4 -2 N_{n,r}^{-1} \sum_{ |I|=|J|=r}  |\Det U[I,J]|^2 + N_{n,r}^{-2} \sum_{ |I|=|J|=r}   1\\
&=\sum_{ |I|=|J|=r}  |\Det U[I,J]|^4 -1. 
\end{align*}
\item Left and right multiplication by diagonal unitary matrices changes each minor only by a scalar of modulus one, while multiplication by permutation matrices only permutes the minors. Hence the collection $\{|\Det U[I,J]|:|I|=|J|=r\}$ is unchanged up to permutation, and so is $\delta_r(U)$.
\end{enumerate}
\end{proof}

For later use, we shall say that two complex Hadamard matrices are \emph{Hadamard equivalent} if one can be obtained from the other by left and right multiplication by diagonal unitary matrices and permutation matrices (see e.g. \cite{Haa97, TZ06}). By Lemma~\ref{lem:C} (4), $\delta_r$ is invariant under Hadamard equivalence.

\subsection{Monotonicity of the uniform-minor discrepancy}
In this section, we prove the monotonicity of $\delta_r$ with respect to $r$.
By Lemma~\ref{lem:C}, it suffices to investigate $\delta_r$ for $r=1, 2,\dots, \lfloor n/2\rfloor$. We define
$$
\Delta_r(U):=P_r(U)-\frac{1}{N_{n,r}}J_{N_{n,r}}.
$$
By definition, we get $\delta_r(U)=\|\Delta_r(U)\|_F^2$.

\begin{thm}\label{thm:monotonicity}
For every $n\ge 4$ and $U\in {\cal U}_n$,
$$
\delta_1(U)\le \delta_2(U) \le\cdots \le \delta_{\lfloor \frac{n}{2} \rfloor} (U),
$$
together with $\delta_r(U)=\delta_{n-r}(U)$ for all $0\le r \le n$.
\end{thm}
\begin{proof}
The symmetry was established in Lemma~\ref{lem:C} (2). We first observe that
\begin{align}\label{eq:DJ}
D_r\left( \frac{1}{N_{n,r}}J_{N_{n,r}}\right)=\left(\frac{1}{N_{n,r-1}}J_{N_{n,r-1}}\right)D_r,
\end{align}
where $D_r$ was defined in \eqref{eq:D}. Indeed, each $(r-1)$-element subset of $[n]$ is contained in exactly $n-r+1$ subsets of cardinality $r$, and hence every entry of $D_r( \frac{1}{N_{n,r}}J_{N_{n,r}})$ is equal to $(n-r+1) N_{n,r}^{-1}$. On the other hand, each $r$-element subset contains exactly $r$ subsets of cardinality $r-1$, and hence every entry of $(\frac{1}{N_{n,r-1}}J_{N_{n,r-1}})D_r$ is equal to $r N_{n,r-1}^{-1}$. Thus, the desired identity \eqref{eq:DJ} follows from $(n-r+1) N_{n,r}^{-1} = r N_{n,r-1}^{-1}$. Combining this identity \eqref{eq:DJ} with Lemma~\ref{lem:DP}, we obtain
\begin{align}\label{eq:DD}
D_r\Delta_r(U)=\Delta_{r-1}(U)D_r.
\end{align}
Applying the same identity to $U^*$ and using $P_r(U^*)=P_r(U)^T$, we obtain $D_r\Delta_r(U)^T=\Delta_{r-1}(U)^TD_r$. Taking the transpose gives
\begin{align}\label{eq:DD2}
\Delta_r(U)D_r^*=D_r^*\Delta_{r-1}(U).
\end{align}
It follows from \eqref{eq:DD} and \eqref{eq:DD2} that
$$
D_rD_r^*\Delta_{r-1}(U)=D_r \Delta_r(U)D_r^\ast=\Delta_{r-1}(U)D_rD_r^*.
$$
Thus,
\begin{align}\label{eq:DD3}
[D_rD_r^*,\Delta_{r-1}(U)]=0.
\end{align}

Suppose that $2\le r\le \left\lfloor\frac{n}{2}\right\rfloor$. Counting the entries gives
$D_r D_r^\ast  = D_{r-1}^\ast  D_{r-1} + (n-2r+2) I_{N_{n,r-1}}$. Since $D_{r-1}^*D_{r-1}$ is positive semidefinite and $n-2r+2>0$, the matrix $D_rD_r^*$ is positive definite. We may define
$C_r:=(D_rD_r^*)^{-1/2}D_r$. Then
$$
C_rC_r^*=(D_rD_r^*)^{-1/2}D_rD_r^*(D_rD_r^*)^{-1/2}=I_{N_{n,r-1}},
$$
so that $C_r$ is a coisometry. Since $\Delta_{r-1}(U)$ commutes with $D_rD_r^*$ by \eqref{eq:DD3}, it also commutes with $(D_rD_r^*)^{-1/2}$. Therefore, by \eqref{eq:DD},
\begin{align}
C_r\Delta_r(U)
&=(D_rD_r^*)^{-1/2}D_r\Delta_r(U)\\
&=(D_rD_r^*)^{-1/2}\Delta_{r-1}(U)D_r\\
&=\Delta_{r-1}(U)(D_rD_r^*)^{-1/2}D_r=\Delta_{r-1}(U)C_r \label{eq:CD}.
\end{align}
Similarly, by \eqref{eq:DD2},
\begin{align}\label{eq:CD2}
\Delta_r(U)C_r^*=C_r^*\Delta_{r-1}(U).
\end{align}
Since $C_rC_r^*=I_{N_{n,r-1}}$, the space $\mathbb{C}^{N_{n,r}}$ admits the orthogonal decomposition
$$
\mathbb{C}^{N_{n,r}}=\Ran C_r^*\oplus\Ker  C_r.
$$
By \eqref{eq:CD}, $\Ker  C_r$ is invariant under $\Delta_r(U)$, while by \eqref{eq:CD2}, $\Ran C_r^*$ is also invariant under $\Delta_r(U)$. Moreover, $C_r:\Ran C_r^*\longrightarrow\mathbb{C}^{N_{n,r-1}}$ is unitary, with inverse $C_r^*$. By \eqref{eq:CD} and \eqref{eq:CD2}, the restriction of $\Delta_r(U)$ to $\Ran C_r^*$ is unitarily equivalent to $\Delta_{r-1}(U)$. Consequently, there exists an operator $R_r(U)$ on $\Ker  C_r$ such that
\begin{align}\label{eq:iso}
\Delta_r(U)\simeq\Delta_{r-1}(U)\oplus R_r(U).
\end{align}
Taking the Frobenius norm in \eqref{eq:iso}, we obtain
\begin{align*}
\delta_r(U)
&=\|\Delta_r(U)\|_F^2\\
&=\|\Delta_{r-1}(U)\|_F^2+\|R_r(U)\|_F^2\\
&=\delta_{r-1}(U)+\|R_r(U)\|_F^2.
\end{align*}
Therefore, $\delta_{r-1}(U)\le \delta_r(U)$ for all $2\le r \le \left\lfloor n/2\right\rfloor$.
This completes the proof.
\end{proof}

\subsection{A finite free interpretation of the discrepancy}

We next make this relation quantitative and show that $\delta_r(U)$ has a direct interpretation in terms of finite free multiplicative convolution.

For $A\in {\cal M}_n$, we can write
$$
\chi_A(x)=\sum_{r=0}^n(-1)^r {\sf e}_r(A)x^{n-r}, \quad \text{ where} \quad
{\sf e}_r(A)=\sum_{|I|=r}\Det A[I,I],
$$
and ${\sf e}_0(A):=1$. In particular, if $D=\diag(d_1,\ldots,d_n)\in{\cal D}_n$, then ${\sf e}_r(D)=\sum_{|I|=r}\prod_{i\in I}d_i$.

Recall that the finite free multiplicative convolution satisfies
$$
(\chi_A\boxtimes_n\chi_B)(x)=\sum_{r=0}^n(-1)^r\frac{{\sf e}_r(A){\sf e}_r(B)}{N_{n,r}}x^{n-r}
$$
for $A,B\in{\cal M}_n$. The following identity shows that the terms appearing in Definition~\ref{defn:UMD} are precisely the coefficients of the error from finite free multiplicative convolution.

\begin{lem}
\label{lem:multiplicative-error}
Let $U\in {\cal U}_n$ and $D=\diag(d_1,\ldots,d_n), E=\diag(e_1,\ldots,e_n) \in {\cal D}_n$. Then, for every $0\leq r\leq n$,
\begin{align*}
{\sf e}_r(DUEU^*)-\frac{{\sf e}_r(D){\sf e}_r(E)}{N_{n,r}}=
\sum_{|I|=|J|=r}
\left(|\Det U[I,J]|^2-\frac{1}{N_{n,r}}\right)d_Ie_J,
\end{align*}
where $d_I:=\prod_{i\in I}d_i$ and $e_J:=\prod_{j\in J}e_j$.
\end{lem}

\begin{proof}
Since $D$ is diagonal, we obtain $\Det(DUEU^*)[I,I]=d_I\Det(UEU^*)[I,I]$.
By the Cauchy--Binet formula, we have
$$
{\sf e}_r(DUEU^*)=\sum_{|I|=r}\Det(DUEU^*)[I,I]=\sum_{|I|=|J|=r}|\Det U[I,J]|^2d_Ie_J.
$$
On the other hand,
$$
\frac{{\sf e}_r(D){\sf e}_r(E)}{N_{n,r}}=\sum_{|I|=|J|=r}\frac{1}{N_{n,r}}d_Ie_J.
$$
Subtracting the two identities proves the assertion.
\end{proof}

Thus, for each $r$, the matrix
$$
\left(|\Det U[I,J]|^2-\frac{1}{N_{n,r}}\right)_{|I|=|J|=r}
$$
may be regarded as the error kernel for the $r$-th coefficient of finite free multiplicative convolution. 

We now obtain an $L^2$ interpretation of the uniform-minor discrepancy.
Let $z_1,\dots,z_n$ and $w_1,\dots,w_n$ be mutually independent random variables, each distributed according to the normalized Haar measure on $\mathbb{T}$. Define
$$
E_z:=\diag(z_1,\dots,z_n)
\qquad \text{and} \qquad
E_w:=\diag(w_1,\dots,w_n).
$$

\begin{thm}
\label{thm:L2-interpretation}
For every $U\in {\cal U}_n$ and $0\leq r\leq n$,
$$
\delta_r(U)=\mathbb{E}\left|{\sf e}_r(E_zUE_wU^*)-\frac{{\sf e}_r(E_z){\sf e}_r(E_w)}{N_{n,r}}\right|^2.
$$
Equivalently,
$$
\delta_r(U)= \left\|{\sf e}_r(E_zUE_wU^*)-\frac{{\sf e}_r(E_z){\sf e}_r(E_w)}{N_{n,r}}\right\|_{L^2(\mathbb{T}^{2n})}^2.
$$
\end{thm}
\begin{proof}
For $I,J\subset[n]$ with $|I|=|J|=r$, we put $z_I:=\prod_{i\in I}z_i$ and $w_J:=\prod_{j\in J}w_j$.
By Lemma~\ref{lem:multiplicative-error},
\begin{align*}
{\sf e}_r(E_zUE_wU^*)-\frac{{\sf e}_r(E_z){\sf e}_r(E_w)}{N_{n,r}}=
\sum_{|I|=|J|=r}\left(|\Det U[I,J]|^2-\frac{1}{N_{n,r}}\right)z_Iw_J.
\end{align*}
Since $z_1,\ldots,z_n$ are independent Haar random variables on $\mathbb{T}$, we have $\mathbb{E}[z_I\overline{z_{I'}}]= \delta_{I,I'}$. Similarly, we obtain $\mathbb{E}[w_J\overline{w_{J'}}]=\delta_{J,J'}$.
Hence $\{z_Iw_J:|I|=|J|=r\}$ is an orthonormal family in $L^2(\mathbb{T}^{2n})$. Therefore,
\begin{align*}
\mathbb{E}
\left|{\sf e}_r(E_zUE_wU^*)-\frac{{\sf e}_r(E_z){\sf e}_r(E_w)}{N_{n,r}}
\right|^2 
&=\sum_{|I|=|J|=r}\left(|\Det U[I,J]|^2-\frac{1}{N_{n,r}}\right)^2=\delta_r(U).
\end{align*}
\end{proof}

Theorem~\ref{thm:L2-interpretation} gives a direct finite free interpretation of the uniform-minor discrepancy. Namely, $\delta_r(U)$ is precisely the mean-square error in the $r$-th elementary symmetric coefficient between $E_zUE_wU^\ast$, with fixed relative position $U$, and the corresponding finite free multiplicative convolution.


\section{An explicit formula for the discrepancy $\delta_2$}\label{sec5}

We now restrict our attention to flat unitaries $U\in\mathcal {\cal U}_n^{\rm flat}$. By Lemma~\ref{lem:C} (1), we get $\delta_1(U)=0$. On the other hand, Theorem~\ref{thm:MASA} shows that, for $n\geq4$, the full uniform-minor condition cannot hold. Moreover, $\delta_2(U)$ is the first nontrivial discrepancy in the hierarchy
$$
0=\delta_1(U)\le \delta_2(U)\le \cdots \le \delta_{\lfloor n/2\rfloor}(U),
$$
by Theorem~\ref{thm:monotonicity}. It is therefore natural to study $\delta_2(U)$ explicitly for $U\in {\cal U}^{\rm flat}_n$. We begin with the following elementary identity.

\begin{lem}\label{lem:D}
For any $z_1,\dots, z_n \in \mathbb{T}$ with $\sum_{j=1}^n z_j =0$, we have
\begin{align*}
\sum_{1\le j<k\le n} |z_j-z_k|^4 = 3n^2 + \left|\sum_{j=1}^n z_j^2\right|^2.
\end{align*}
\end{lem}
\begin{proof}
An elementary computation shows that
\begin{align*}
\sum_{1\le j<k \le n} (z_j\overline{z_k} +\overline{z_j}z_k) = \left| \sum_{j=1}^n z_j\right|^2 - n=-n
\end{align*}
and
\begin{align*}
\sum_{1\le j<k \le n} \{(z_j\overline{z_k})^2 +(\overline{z_j}z_k)^2\} = \left| \sum_{j=1}^n z_j^2 \right|^2 - n.
\end{align*}
Consequently, we have
\begin{align*}
\sum_{1\le j<k\le n} |z_j-z_k|^4 
&= \sum_{1\le j<k\le n} \{ 6 - 4(z_j\overline{z_k} + \overline{z_j}z_k) + ((z_j\overline{z_k})^2 + (\overline{z_j}z_k)^2)\}\\
&=6 \binom{n}{2} -4 (-n) + \left|\sum_{j=1}^n z_j^2\right|^2 -n\\
&=3n^2+\left|\sum_{j=1}^n z_j^2\right|^2.
\end{align*}
\end{proof}

\begin{thm}\label{thm:UMD}
Assume that $n\ge 4$ and $U\in  {\cal U}_n^{\rm flat}$ and write $\sqrt{n}U=(h_{ij}) \in {\cal H}_n$. Then we have
\begin{align*}
\delta_2(U) &= \frac{n-3}{2n} + \frac{1}{n^4} \sum_{1\le p<q \le n} \left|\sum_{j=1}^n h_{pj}^2 \overline{h_{qj}^2}\right|^2 \\
&= \frac{n-3}{2n} + \frac{1}{2n^4}\| (\sqrt{n}U)^{\circ 2} ((\sqrt{n}U)^{\circ 2})^\ast - n I_n\|_F^2.
\end{align*}
Consequently, we have 
$$
\delta_2(U)\ge \frac{n-3}{2n},
$$ 
with equality if and only if $(\sqrt{n}U)^{\circ 2}\in {\cal H}_n$.
\end{thm}
\begin{proof}
As in the proof of (2) $\Rightarrow$ (4) in Theorem~\ref{thm:MASA}, for any $1\le p<q\le n$ and $1\le j<k \le n$, we obtain
$$
|\Det U[\{p,q\},\{j,k\}]|^2 = \frac{1}{n^2} \left|h_{pj}\overline{h_{qj}}- h_{pk}\overline{h_{qk}}\right|^2.
$$
For fixed $p<q$, set $z_j := h_{pj}\overline{h_{qj}}$. Then $|z_j|=1$ for all $1\le j \le n$ and $\sum_{j=1}^n z_j = 0$. Therefore, for fixed $p<q$, we have
\begin{align*}
\sum_{j<k} |\Det U[\{p,q\}, \{j,k\}]|^4 
&= \frac{1}{n^4} \sum_{j<k}|z_j-z_k|^4\\
&= \frac{1}{n^4} \left\{ 3n^2 + \left| \sum_{j=1}^n z_j^2\right|^2\right\} \qquad \text{(by Lemma~\ref{lem:D})}\\
&=\frac{3}{n^2} + \frac{1}{n^4}\left| \sum_{j=1}^n h_{pj}^2 \overline{h_{qj}^2}\right|^2.
\end{align*}
Consequently, we obtain
\begin{align*}
\sum_{|I|=|J|=2} |\Det U[I,J]|^4= \frac{3(n-1)}{2n}+ \frac{1}{n^4} \sum_{1\le p<q\le n} \left| \sum_{j=1}^n h_{pj}^2 \overline{h_{qj}^2}\right|^2.
\end{align*}
By Lemma~\ref{lem:C} (3), we get
\begin{align*}
\delta_2(U) 
&= \frac{3(n-1)}{2n}+ \frac{1}{n^4} \sum_{1\le p<q\le n} \left| \sum_{j=1}^n h_{pj}^2 \overline{h_{qj}^2}\right|^2 -1\\
&=\frac{n-3}{2n} + \frac{1}{n^4} \sum_{1\le p<q \le n} \left|\sum_{j=1}^n h_{pj}^2 \overline{h_{qj}^2}\right|^2.
\end{align*}
The Frobenius-norm identity in the statement follows by summing the squared absolute values of the off-diagonal entries. Clearly, we have $\delta_2(U)\ge \frac{n-3}{2n}$. Moreover, $\delta_2(U)= \frac{n-3}{2n}$ if and only if $\sum_{j=1}^n h_{pj}^2 \overline{h_{qj}^2}=0$ for any $1\le p<q\le n$, equivalently, $(\sqrt{n}U)^{\circ 2} ((\sqrt{n}U)^{\circ 2})^\ast =n I_n$. Since $|h_{pj}|=1$, we also have $|h_{pj}^2|=1$ for all $1\le p,j\le n$. Finally, this means that $(\sqrt{n}U)^{\circ 2}\in {\cal H}_n$.
\end{proof}

As a consequence of Theorems~\ref{thm:monotonicity} and \ref{thm:UMD}, we obtain the following lower bound for the uniform-minor discrepancy.
\begin{cor}
For any $U\in  {\cal U}_n^{\rm flat}$, we have
$$
\delta_r(U) \ge \frac{n-3}{2n}, \qquad 2 \le r \le n-2.
$$
\end{cor}

A complex Hadamard matrix whose entrywise square is also complex Hadamard is called an {\it S-Hadamard matrix}; see \cite[Definition 2.1]{L19}. Thus, equality in Theorem~\ref{thm:UMD} holds precisely when $\sqrt{n}U$ is an S-Hadamard matrix.

S-Hadamard matrices exist in every odd order: the Fourier matrix of odd order is S-Hadamard; see \cite[Proposition 2.3.7]{P26}. Hence the lower bound in Theorem~\ref{thm:UMD} is attained in every odd dimension $n\ge 5$. For even orders, S-Hadamard matrices are known to exist in several cases (see \cite[Proposition 2.3 and Example 2.4]{L19}), but the existence question remains open in general. This leads naturally to the problem of determining the smallest possible uniform-minor discrepancy in each dimension.

For $n\ge 4$, we define
$$
m_n:=\min_{U\in{\cal U}_n^{\rm flat}}\delta_2(U).
$$
The minimum exists because ${\cal U}_n^{\rm flat}$ is nonempty and compact and $\delta_2$ is continuous. The equality characterization in Theorem~\ref{thm:UMD} immediately gives the following proposition.

\begin{prop}\label{prop:S-Hadamard}
For $n\ge 4$, the following conditions are equivalent.
\begin{enumerate}[\rm (1)]
\item There exists an S-Hadamard matrix of order $n$.
\item $m_n=\dfrac{n-3}{2n}$.
\end{enumerate}
\end{prop}
\begin{proof}
If there exists an S-Hadamard matrix $H$ of order $n$, then $U=\frac{1}{\sqrt{n}}H \in {\cal U}_n^{\rm flat}$ and $\delta_2(U)=\frac{n-3}{2n}$ by Theorem~\ref{thm:UMD}. Conversely, if $m_n = \frac{n-3}{2n}$, then there exists $U\in {\cal U}_n^{\rm flat}$ such that $\delta_2(U)=\frac{n-3}{2n}$. By Theorem~\ref{thm:UMD}, $\sqrt{n} U$ is an S-Hadamard matrix of order $n$.
\end{proof}

If no S-Hadamard matrix of order $n$ exists, then $m_n>\frac{n-3}{2n}$. In this case, determining $m_n$ gives the optimal improvement of the lower bound in Theorem~\ref{thm:UMD}.

\begin{ex}[The class ${\cal U}^{\rm flat}_4$]\label{ex:S-Hadamard4}
By the classification in \cite[Section 5.4]{TZ06}, every matrix in ${\cal U}_4^{\rm flat}$ is Hadamard equivalent to
$$
U(t)=\frac{1}{2}
\begin{pmatrix}
1 & 1 & 1 & 1\\
1 & e^{{\rm i}t} & -1 & -e^{{\rm i}t}\\
1 & -1 & 1 & -1\\
1 & -e^{{\rm i}t} & -1 & e^{{\rm i}t}
\end{pmatrix},
\qquad t\in\mathbb R.
$$
Since $\delta_2$ is invariant under Hadamard equivalence, it is enough to compute $\delta_2(U(t))$. A direct computation shows that
\begin{align*}
(\sqrt{4} U(t))^{\circ 2} ((\sqrt{4} U(t))^{\circ 2} )^\ast -4 I_4 
= 2 \begin{pmatrix}
0 & 1+ e^{-2{\rm i}t} & 2 & 1 + e^{-2{\rm i}t}\\
1 + e^{2{\rm i}t} & 0 & 1+ e^{2{\rm i}t} & 2\\
2 & 1+ e^{-2{\rm i}t} & 0 & 1+ e^{-2{\rm i}t}\\
1 + e^{2{\rm i}t} & 2 & 1+e^{2{\rm i}t} & 0
\end{pmatrix}.
\end{align*}
Therefore, we have
\begin{align*}
\|(\sqrt{4} U(t))^{\circ 2} ((\sqrt{4} U(t))^{\circ 2} )^\ast -4 I_4\|_F^2
&= 4 \{ 4 |1+e^{2{\rm i}t}|^2 + 4 |1+e^{-2{\rm i}t}|^2 + 4 \times 2^2\}\\
&=64 (2 \cos^2 t +1).
\end{align*}
Using Theorem \ref{thm:UMD}, we obtain
\begin{align*}
\delta_2(U(t))=\frac{1}{8} + \frac{1}{2\times 4^4} \times 64 (2 \cos^2 t +1)=\frac{1+\cos^2 t}{4}.
\end{align*}
Therefore, $\min_{t\in\mathbb R}\delta_2(U(t))=1/4$ and the minimum is attained precisely when $\cos t=0$. Since every matrix in ${\cal U}^{\rm flat}_4$ is Hadamard equivalent to some
$U(t)$ and $\delta_2$ is invariant under Hadamard equivalence, we conclude that $m_4=1/4$. 

On the other hand, Theorem~\ref{thm:UMD} gives the universal lower bound $1/8$. Hence $m_4>1/8$, so the universal lower bound is not attained in order $4$. Thus, there is no S-Hadamard matrix of order $4$. This nonexistence is also recorded in \cite[Proposition 2.3.6]{P26} and follows from the classification of order $4$ complex Hadamard matrices in \cite[Section 5.4]{TZ06}.
\end{ex}

\begin{ex}[Fourier matrices]\label{ex:Fourier}
Let $n \ge 4$, and let
$$
F_n=\left(\omega_n^{(p-1)(j-1)}\right)_{p,j=1}^n,
\qquad \omega_n=e^{2\pi {\rm i}/n},
$$
be the $n\times n$ Fourier matrix, and put $U_n:=\frac1{\sqrt n}F_n\in {\cal U}_n^{\rm flat}$; see \cite{TZ06} for details. We compute $\delta_2(U_n)$ explicitly. By Theorem~\ref{thm:UMD}, we have
$$
\delta_2(U_n)=\frac{n-3}{2n}+\frac{1}{n^4}
\sum_{1\leq p<q\leq n}
\left|\sum_{j=1}^n
h_{pj}^2\overline{h_{qj}^2}\right|^2,
$$
where $h_{pj}=\omega_n^{(p-1)(j-1)}$. For $p<q$,
$$
\sum_{j=1}^n
h_{pj}^2\overline{h_{qj}^2}
=\sum_{j=0}^{n-1}\omega_n^{2(p-q)j}
=\begin{cases}
n, & n\mid 2(p-q),\\
0, & n\nmid 2(p-q).
\end{cases}
$$

If $n$ is odd, then $n\mid 2(p-q)$ never occurs since $1\le p < q \le n$. Hence
$$
\delta_2(U_n)=\frac{n-3}{2n}.
$$

If $n$ is even, then the condition $n\mid 2(p-q)$ holds precisely when $q-p=n/2$. There are exactly $n/2$ such pairs $(p,q)$, and hence
$$
\sum_{1\leq p<q\leq n}
\left|\sum_{j=1}^nh_{pj}^2\overline{h_{qj}^2}\right|^2
=\frac{n}{2}\times n^2
=\frac{n^3}{2}.
$$
Therefore
$$
\delta_2(U_n)=\frac{n-3}{2n}+\frac{1}{2n}=\frac{n-2}{2n}.
$$
Consequently, we get
$$
\delta_2\left(\frac1{\sqrt n}F_n\right)=
\begin{cases}
\dfrac{n-3}{2n},
& n \text{ is odd},\\[3mm]
\dfrac{n-2}{2n},
& n \text{ is even}.
\end{cases}
$$
In particular, when $n$ is odd, $U_n$ attains the universal lower bound in Theorem~\ref{thm:UMD}, and $F_n$ is an S-Hadamard matrix.
\end{ex}

\begin{rem}
The nonexistence of an S-Hadamard matrix of order $8$ is conjectured in \cite[Conjecture 3.1.8]{P26}. By Proposition~\ref{prop:S-Hadamard}, this conjecture is equivalent to $m_8>5/16$. Example \ref{ex:Fourier} gives the upper bound
$$
\frac{5}{16}\le m_8\le\frac{3}{8}.
$$
In particular, determining whether $m_8>5/16$ is equivalent to settling the existence question in order $8$.
\end{rem}

Finally, we compute the discrepancy $\delta_2$ of tensor products of flat unitary matrices.

\begin{prop}\label{prop:tensor}
Let $m,n\ge 4$, $U \in  {\cal U}_m^{\rm flat}$ and $V\in  {\cal U}_n^{\rm flat}$. Then we have
\begin{align*}
\delta_2(U\otimes V) = &  \frac{mn-3}{2mn}+ 2 \left( \delta_2(U)-\frac{m-3}{2m}\right)\left(\delta_2(V)-\frac{n-3}{2n}\right) \\
&+\frac{1}{n}\left( \delta_2(U)-\frac{m-3}{2m}\right) + \frac{1}{m}\left(\delta_2(V)-\frac{n-3}{2n}\right).
\end{align*}
\end{prop}

\begin{proof}
Put $H=\sqrt{m}U\in{\cal H}_m$ and $K=\sqrt{n}V\in{\cal H}_n$, and define $R=H^{\circ 2}(H^{\circ 2})^*$ and $S=K^{\circ 2}(K^{\circ 2})^*$. Since the diagonal entries of $R$ and $S$ are $m$ and $n$, respectively, we have ${\sf Tr}(R)=m^2$ and ${\sf Tr}(S)=n^2$. Consequently, we get
$$
\|R-mI_m\|_F^2=\|R\|_F^2-m^3, \qquad \|S-nI_n\|_F^2=\|S\|_F^2-n^3.
$$
The matrix $H\otimes K$ is complex Hadamard and $(H\otimes K)^{\circ 2}\bigl((H\otimes K)^{\circ 2}\bigr)^*=R\otimes S$. Using the multiplicativity of the trace and the Frobenius norm
under tensor products, we obtain
\begin{align*}
\|R\otimes S-mnI_{mn}\|_F^2
&=\|R\|_F^2\|S\|_F^2-m^3n^3\\
&=\bigl(\|R-mI_m\|_F^2+m^3\bigr)\bigl(\|S-nI_n\|_F^2+n^3\bigr)-m^3n^3\\
&=\|R-mI_m\|_F^2\|S-nI_n\|_F^2+n^3\|R-mI_m\|_F^2+m^3\|S-nI_n\|_F^2.
\end{align*}
By Theorem \ref{thm:UMD}, we have
$$
\|R-mI_m\|_F^2=2m^4\left(\delta_2(U)-\frac{m-3}{2m}\right),
\qquad
\|S-nI_n\|_F^2=2n^4\left(\delta_2(V)-\frac{n-3}{2n}\right).
$$
Applying the same theorem to $U\otimes V$ therefore gives
\begin{align*}
\delta_2(U\otimes V)
&=\frac{mn-3}{2mn}+\frac{1}{2m^4n^4}\|R\otimes S-mnI_{mn}\|_F^2\\
&=\frac{mn-3}{2mn}+2\left(\delta_2(U)-\frac{m-3}{2m}\right)\left(\delta_2(V)-\frac{n-3}{2n}\right)\\
&\quad+\frac{1}{n}\left(\delta_2(U)-\frac{m-3}{2m}\right)+\frac{1}{m}\left(\delta_2(V)-\frac{n-3}{2n}\right),
\end{align*}
as required.
\end{proof}

\subsection*{Acknowledgement}
The author was supported by JSPS KAKENHI Grant Number JP22K13925. 

\subsection*{Declaration on the use of Generative AI}
ChatGPT 5.6 was used as an exploratory aid for mathematical discussions, particularly concerning the uniform-minor discrepancy, for literature searches on Hadamard matrices, and for English-language editing and suggestions on wording and presentation. All mathematical results, proofs, and calculations in the final manuscript were independently verified by the author.

\vspace{6mm}

\hspace{-7mm}{\bf Yuki Ueda}:\\
Faculty of Education, Hokkaido University of Education\\
9 Hokumon-cho, Asahikawa, Hokkaido 070-8621, Japan\\
E-mail: ueda.yuki@a.hokkyodai.ac.jp

\end{document}